\documentclass[12pt]{amsart}

\usepackage{amsmath}

\usepackage{amsfonts,amssymb,amsmath,latexsym,wasysym,mathrsfs,bbm,stmaryrd}
\usepackage[all]{xy}
\begin{document}

\newtheorem{theorem}{Theorem}[section]
\newtheorem{lemma}[theorem]{Lemma}
\newtheorem{proposition}[theorem]{Proposition}
\newtheorem{corollary}[theorem]{Corollary}

\theoremstyle{definition}
\newtheorem{definition}[theorem]{Definition}
\newtheorem{example}[theorem]{Example}
\newtheorem{xca}[theorem]{Exercise}
\newtheorem{notation}[theorem]{Notation}

\theoremstyle{remark}
\newtheorem{remark}[theorem]{Remark}

\numberwithin{equation}{section}

\newcommand{\A}{\mathcal A}
\newcommand{\E}{ {\it E} }
\newcommand{\tr}{\text{\it{tr}}}
\newcommand{\Tr}{\text{\it{Tr}}}
\newcommand{\ff}{\varphi}
\newcommand{\CC}{\mathbb C}
\newcommand{\RR}{\mathbb R}
\newcommand{\NN}{\mathbb N}
\newcommand{\cP}{\mathcal{P}}
\newcommand{\LzRn}{{L^2(\RR_+^n)}}
\newcommand{\LzRm}{{L^2(\RR_+^m)}}
\newcommand{\LzRz}{{L^2(\RR_+^2)}}
\newcommand{\FF}{\mathbb F}
\newcommand{\la}{\langle}
\newcommand{\ra}{\rangle}

\newcommand{\kk}{\kappa}
\newcommand{\oseins}{\overset{1}\frown}
\newcommand{\tensor}{\otimes}

\begin{center}
{\Large{\bf Multidimensional semicircular limits on the free Wigner chaos}}
\normalsize
\\~\\ by Ivan Nourdin\footnote{Universit\'e Nancy 1, France. Email: inourdin@gmail.com}, 
Giovanni Peccati\footnote{Universit\'e du Luxembourg, Luxembourg. Email: giovanni.peccati@gmail.com}, and 
Roland Speicher\footnote{Universit\"at des Saarlandes, Germany. Email: speicher@math.uni-sb.de}
\end{center}

\bigskip

{\small \noindent {{\bf Abstract}: We show that, for sequences of vectors of multiple 
Wigner integrals with respect to a free Brownian motion, 
componentwise convergence to semicircular is equivalent to joint convergence. 
This result extends to the free probability setting some findings by Peccati and Tudor (2005), and represents a 
multidimensional counterpart of a limit theorem inside the free Wigner chaos established by 
Kemp, Nourdin, Peccati and Speicher (2011).}
\\

\noindent {\bf Key words}: Convergence in Distribution; Fourth Moment Condition; Free Brownian Motion; Free Probability; Multidimensional Limit Theorems; Semicircular Law; Wigner Chaos.  \\

\noindent {\bf 2000 Mathematics Subject Classification:} 46L54, 60H05, 60H07, 60H30.

\section{Introduction}

Let $W=\{W_t :t\geq 0\}$ be a one-dimensional standard Brownian motion (living on some probability space 
$(\Omega, \mathscr{F}, P)$). For every $n\geq 1$ and every real-valued, symmetric and square-integrable function $f\in 
L^2(\mathbb{R}_+^n)$, we denote by $I^W(f)$ the multiple Wiener-It\^o integral of $f$, with respect to $W$. Random variables of 
this type compose the so-called $n$th {\it Wiener chaos} associated with $f$. In an infinite-dimensional setting, the concept 
of Wiener chaos plays the same role as that of the Hermite polynomials for the one-dimensional Gaussian distribution, and 
represents one of the staples of modern Gaussian analysis (see e.g. \cite{Janson, NP-book, nualart-book, PecTaqbook} for an 
introduction to these topics).

In recent years, many efforts have been made in order to characterize Central Limit Theorems (CLTs) -- that is, limit theorems 
involving convergence in distribution to a Gaussian element -- for random variables living inside a Wiener chaos.  The following 
statement gathers the main findings of \cite{nunugio} (Part 1) and \cite{pectud} (Part 2), and provides a complete 
characterization of (both one- and multi-dimensional) CLTs on the Wiener chaos.

\begin{theorem}[See \cite{nunugio, pectud}]\label{t:NPPT} 

\begin{itemize}
\item[\rm (A)]Let $F_k = I^W(f_k)$, $k\geq 1$, be a sequence of multiple integrals of order $n\geq 2$, such that 
$E[F_k^2]\to 1$. Then, the following two assertions are equivalent, as $k\to \infty$: {\rm\bf (i)} $F_k$ converges in 
distribution to a standard Gaussian random variable $N\sim \mathscr{N}(0,1)$; {\rm\bf (ii)} $E[F_k^4]\to 
3 =E[N^4]$. 

\item[\rm (B)] Let $d\geq 2$ and $n_1,...,n_d$ be integers, and let $(F_k^{(1)},...,F_k^{(d)})$, $k\geq 1$, be a sequence of 
random vectors such that, for every $i=1,...,d$, the random variable $F_k^{(i)}$ lives in the $n_i$th Wiener chaos of $W$. 
Assume that, as $k\to \infty$ and for every $i,j=1,...,d$, $E[F_k^{(i)}F_k^{(j)}]\to c(i,j)$, where $c = \{c(i,j) : 
i,j=1,...,d\}$ is a positive definite symmetric matrix. Then, the following two assertions are equivalent, as $k\to \infty$: 
{\rm\bf (i)} $(F_k^{(1)},...,F_k^{(d)})$ converges in distribution to a centered $d$-dimensional Gaussian vector 
$(N_1,...,N_d)$ with covariance $c$; {\rm\bf (ii)} for every $i=1,...,d$, $F_k^{(i)}$ converges in distribution to a centered 
Gaussian random variable with variance $c(i,i)$.
\end{itemize}
\end{theorem}

Roughly speaking, Part (B) of the previous statement means that, for vectors of random variables living inside some fixed 
Wiener chaoses, {\it componentwise convergence to Gaussian always implies joint convergence}. The combination of Part (A) 
and Part (B) of Theorem \ref{t:NPPT} represents a powerful simplification of the so-called `method of moments and cumulants' 
(see e.g. \cite[Chapter 11]{PecTaqbook} for a discussion of this point), and has triggered a considerable number of 
applications, refinements and generalizations, ranging from Stein's method to analysis on homogenous spaces, random matrices 
and fractional processes -- see the survey \cite{NP-survey} as well as the forthcoming monograph \cite{NP-book} for details 
and references. 

Now, let $(\mathscr{A},\ff)$ be a non-commutative tracial $W^*$-probability space (in particular, $\mathscr{A}$ is a von 
Neumann algebra and $\varphi$ is a trace -- se Section \ref{ss:free} for details), and let $S = \{S_t : t\geq 0\}$ be a 
free Brownian motion defined on it. It is well-known (see e.g. \cite{BianeSpeicher}) that, for every $n\geq 1$ and every 
$f\in L^2(\mathbb{R}_+^n)$, one can define a free multiple stochastic integral with respect to $f$. Such an object is usually 
denoted by $I^S(f)$. Multiple integrals of order $n$ 
with respect to $S$ compose the so-called $n$th {\it Wigner chaos} associated with $S$. Wigner chaoses play a fundamental role 
in free stochastic analysis -- see again \cite{BianeSpeicher}.

The following theorem, which is the main result of \cite{KNPS}, is the exact free analogous of Part (A) of Theorem \ref{t:NPPT}.
Note that the value $2$ coincides with the fourth moment of the standard semicircular distribution $S(0,1)$.

\begin{theorem}[See \cite{KNPS}]\label{t:KNPS}
Let $n\ge2$ be an integer, and let $(f_k)_{k\in\mathbb{N}}$ be a sequence of mirror symmetric (see Section \ref{ss:mirror} 
for definitions) functions in $L^2(\mathbb{R}_+^n)$, each with $\|f_k\|_{L^2(\mathbb{R}_+^n)} = 1$.  
The following statements are equivalent.
\begin{itemize}
\item[(1)] The fourth moments of the stochastic integrals $I(f_k)$ converge to $2$, that is,
\[ \lim_{k\to\infty} \ff(I^S(f_k)^4) = 2. \]
\item[(2)] The random variables $I^S(f_k)$ converge in law to the standard semicircular distribution $S(0,1)$ as $k\to\infty$.
\end{itemize}
\end{theorem}

The aim of this 
paper is to provide a complete proof of the following Theorem \ref{t:main}, which represents a free analogous of Part (B) of 
Theorem \ref{t:NPPT}. 

\begin{theorem}\label{t:main}
 Let $d\geq 2$ and $n_1,\dots,n_d$ be some fixed integers, and consider a positive definite symmetric matrix 
$c =\{c(i,j) : i,j=1,...,d\}$. Let $(s_1,\dots,s_d)$ be a semicircular family with covariance $c$ (see Definition 
\ref{d:semifam}). For each $i=1,\dots,d$, we consider a sequence $(f_k^{(i)})_{k\in\NN}$ of mirror-symmetric functions in $L^2(\RR^{n_i}_+)$ such that, for all $i,j=1,\dots,d$,
\begin{equation}\label{thm:assumption}
\lim_{k\to\infty}\varphi[I^S(f_k^{(i)})I^S(f_k^{(j)})] = c(i,j).
\end{equation}
The following three statements are equivalent as $k\to\infty$.\\
(1) The vector $((I^S(f_k^{(1)}),\dots,I^S(f_k^{(d)}))$ converges in distribution to $(s_1,\dots,s_d)$.\\
(2) For each $i=1,\dots,d$, the random variable $I^S(f_k^{(i)})$ converges in distribution to $s_i$.\\
(3) For each $i=1,\dots,d$,
$$\lim_{k\to\infty}\varphi[ I^S(f_k^{(i)})^4]=2\,c(i,i)^2.$$
\end{theorem}

\begin{remark} In the previous statement, the quantity $\varphi[I^S(f_k^{(i)})I^S(f_k^{(j)})]$ 
equals $\la f_k^{(i)},f_k^{(j)}\ra_{L^2(\RR_+^{n_i})}$ if $n_i = n_j$, and equals $ 0$ if $n_i\neq n_j$. In particular, 
the limit covariance matrix $c$ is necessarily such that $c(i,j) = 0$ whenever $n_i\neq n_j$.
\end{remark}

\begin{remark} Two additional references deal with non-semicircular limit theorems inside the free Wigner chaos. 
In \cite{NP-freepoisson}, one can find necessary and sufficient conditions for the convergence towards the so-called 
Mar$\check{\rm c}$enko-Pastur distribution (mirroring analogous findings in the classical setting -- see \cite{NP-gamma}). 
In \cite{deyanourdin}, conditions are established for the convergence towards the so-called `tetilla law' 
(or `symmetric Poisson distribution' -- see also \cite{NicaSpeicherDuke}).

\end{remark}

Combining the content of Theorem \ref{t:main} with those in \cite{KNPS,pectud}, we can finally state the following
Wiener-Wigner transfer principle, establishing an equivalence between multidimensional limit theorems on the classical and free chaoses.

\begin{theorem}\label{t:main2}
 Let $d\geq 1$ and $n_1,\dots,n_d$ be some fixed integers, and consider a positive definite symmetric matrix 
$c =\{c(i,j) : i,j=1,...,d\}$. Let $(N_1,\ldots,N_d)$ be a $d$-dimensional Gaussian vector 
and $(s_1,\dots,s_d)$ be a semicircular family, both with covariance $c$. 
For each $i=1,\dots,d$, we consider a sequence $(f_k^{(i)})_{k\in\NN}$ of fully-symmetric functions (cf. Definition
\ref{ss:mirror}) in 
$L^2(\RR^{n_i}_+)$.
Then:
\begin{enumerate}
\item For all $i,j=1,\ldots,d$ and as $k\to\infty$, $\varphi[I^S(f_k^{(i)})I^S(f_k^{(j)})] \to c(i,j)$
if and only if $E[I^W(f_k^{(i)})I^W(f_k^{(j)})] \to \sqrt{(n_i)!(n_j)!}\,c(i,j)$.
\item If the asymptotic relations in {\rm (1)} are verified then, as $k\to\infty$, 
\[
\big(I^S(f_k^{(1)}),\dots,I^S(f_k^{(d)})\big)\overset{\rm law}{\to}(s_1,\dots,s_d)
\]
if and only if
\[
\big(I^W(f_k^{(1)}),\dots,I^W(f_k^{(d)})\big)\overset{\rm law}{\to}\big(\sqrt{(n_i)!}N_1,\dots,\sqrt{(n_d)!}N_d\big).
\]
\end{enumerate}
\end{theorem}

The remainder of this paper is organized as follows. Section \ref{sec1} gives concise background and notation for the
free probability setting.
Theorems \ref{t:main} and \ref{t:main2} are then proved in Section \ref{s:proofs}.

\section{Relevant definitions and notations}\label{sec1}

We recall some relevant notions and definitions from free stochastic analysis. For more details, we refer the reader to 
\cite{BianeSpeicher, KNPS, NicaSpeicherBook}. 
\subsection{Free probability, free Brownian motion and stochastic integrals}\label{ss:free}

In this note, we consider as given a so-called {\it (tracial) $W^*$ probability space} $(\mathscr{A},\varphi)$, where 
$\mathscr{A}$ is a von Neumann algebra (with involution $X\mapsto X^*$), and $\varphi : \mathscr \to \,\mathbb{C}$ is a 
{\it tracial state} (or {\it trace}). In particular, $\varphi$ is weakly continuous, positive (that is, $\varphi(Y)\geq 0$ 
whenever $Y$ is a nonnegative element of $\mathscr{A}$), faithful (that is, $\varphi(YY^*)=0$ implies $Y = 0$, for every $Y\in \mathscr{A}$) and tracial (that is, $\varphi(XY) = \varphi(YX)$, for every $X,Y\in \mathscr{A}$). The self-adjoint elements of $\mathscr{A} $ are referred to as {\it random variables}. The {\it law} of a random variable $X$ is the unique Borel measure on $\mathbb{R}$ having the same moments as $X$ (see \cite[Proposition 3.13]{NicaSpeicherBook}). For $1\leq p \leq \infty $, one writes $L^p(\mathscr{A},\varphi)$ to indicate the $L^p$ space obtained as the completion of $\mathscr{A}$ with respect to the norm $\| a\|_p = \tau(|a|^p)^{1/p}$, where $ |a|= \sqrt{a^\ast a}$, and $\|\cdot\|_\infty$ stands for the operator norm.

\begin{definition}\label{d:freeness}
Let $\mathscr{A}_1,\ldots,\mathscr{A}_n$ be unital subalgebras of $\mathscr{A}$.  Let $X_1,\ldots, X_m$ be elements chosen from among the $\mathscr{A}_i$'s such that, for $1\le j<m$, $X_j$ and $X_{j+1}$ do not come from the same $\mathscr{A}_i$, and such that $\varphi(X_j)=0$ for each $j$.  The subalgebras $\mathscr{A}_1,\ldots,\mathscr{A}_n$ are said to be {\it free} or {\it freely independent} if, in this circumstance, $\varphi(X_1X_2\cdots X_n) = 0$.  Random variables are called freely independent if the unital algebras they generate are freely independent. 
\end{definition}

\begin{definition}\label{d:semicircular}
The (centered) {\it semicircular distribution} (or Wigner law) $S(0,t)$ is the probability distribution
\begin{equation} \label{eq semicircle} S(0,t)(dx) = \frac{1}{2\pi t} \sqrt{4t-x^2}\,dx, \quad |x|\le 2\sqrt{t}. \end{equation}
Being symmetric around $0$, the odd moments of this distribution are all $0$.  Simple calculations (see e.g. \cite[Lecture 2]{NicaSpeicherBook}) show that the even moments can be expressed in therms of the so-called {\it Catalan numbers}: for non-negative integers $m$,
\[ \int_{-2\sqrt{t}}^{2\sqrt{t}} x^{2m} S(0,t)(dx) = C_m t^m, \]
where $C_m = \frac{1}{m+1}\binom{2m}{m}$ is the $m$th Catalan number.  In particular, the second moment (and variance) is $t$ while the fourth moment is $2t^2$.
\end{definition}

\begin{definition}\label{d:freebm}
A {\it free Brownian motion} $S$ consists of: (i) a filtration $\{\mathscr{A}_t : t\geq 0\}$ of von Neumann sub-algebras of 
$\mathscr{A}$ (in particular, $\mathscr{A}_s \subset \mathscr{A}_t$, for $0\leq s<t$), (ii) a collection 
$S = \{S_t : t\geq 0\}$ of self-adjoint operators in $\mathscr{A}$ such that: (a) $S_0=0$ and $S_t\in\mathscr{A}_t$  
for every $t$, (b) for every $t$, $S_t$ has a semicircular distribution with mean zero and variance $t$, and (c)
for every $0\leq u<t$, the increment $S_t-S_u$ is free with respect to $\mathscr{A}_u$, and has a semicircular distribution with mean zero and variance $t-u$.

\end{definition}

For the rest of the paper, we consider that the $W^*$-probability space $(\mathscr{A}, \varphi)$ is endowed with a free 
Brownian motion $S$. For every integer $n\geq 1$, the collection of all operators having the form of a multiple integral 
$I^S(f)$, $f \in L^2(\mathbb{R}_+^n;\mathbb{C})= L^2(\mathbb{R}_+^n)$, is defined according to \cite[Section 5.3]{BianeSpeicher}, 
namely: (a) first define 
$I^S(f) = (S_{b_1} - S_{a_1})\cdots (S_{b_n} - S_{a_n})$ for every function $f$ having the form
\begin{equation}\label{e:simple}
f(t_1,...,t_n) = {\bf 1}_{(a_1,b_1)}(t_1)\times \ldots \times {\bf 1}_{(a_n,b_n)}(t_n),
\end{equation}
where the intervals $(a_i,b_i)$, $i=1,...,n$, are pairwise disjoint; (b) extend linearly the definition of $I^S(f)$ to 
`simple functions vanishing on diagonals', that is, to functions $f$ that are finite linear combinations of indicators of the 
type (\ref{e:simple}); (c) exploit the isometric relation
\begin{equation}\label{e:freeisometry}
\langle I^S(f),I^S(g) \rangle_{L^2(\mathscr{A},\ff )}=  \int_{\RR_+^n} f(t_1,\ldots,t_n)\overline{g(t_n,\ldots,t_1)}dt_1\ldots dt_n,
\end{equation}
where $f,g$ are simple functions vanishing on diagonals, and use a density argument to define $I(f)$ for a general $f\in  
L^2(\mathbb{R}_+^n)$.

\medskip

As recalled in the Introduction, for $n\geq 1$, the collection of all random variables of the type $I^S(f)$, 
$f\in L^2(\mathbb{R}_+^n)$, is called the $n$th {\it Wigner chaos} associated with $S$. One customarily writes $I^S(a) = a$ 
for every complex number $a$, that is, the Wigner chaos of order $0$ coincides with $\mathbb{C}$. Observe that 
(\ref{e:freeisometry}) together with the above sketched construction imply that, for every $n,m\geq 0$, and every 
$f\in L^2(\mathbb{R}_+^n)$, $g\in L^2(\mathbb{R}_+^m)$,
\begin{equation}\label{iso}
\varphi[I^S(f)I^S(g)] ={\bf 1}_{n=m} \times \int_{\RR_+^n} f(t_1,\ldots,t_n)\overline{g(t_n,\ldots,t_1)}dt_1\ldots dt_n,
\end{equation}
where the right hand side of the previous expression coincides by convention with the inner product in $L^2(\mathbb{R}_+^0) = \mathbb{C}$ whenever $m=n=0$.

\subsection{Mirror Symmetric Functions and Contractions}\label{ss:mirror}
\begin{definition} \label{def mirror} Let $n$ be a natural number, and let $f$ be a function in $L^2(\RR_+^n)$.
\begin{itemize}
\item[(1)] The {\it adjoint} of $f$ is the function $f^\ast(t_1,\ldots,t_n) = \overline{f(t_n,\ldots,t_1)}$.
\item[(2)] $f$ is called {\it mirror symmetric} if $f=f^\ast$, i.e.\ if \[f(t_1,\ldots,t_n) = \overline{f(t_n,\ldots,t_1)}\] for almost all $t_1,\ldots,t_n\ge0$ with respect to the product Lebesgue measure
\item[(3)] $f$ is called {\it fully symmetric} if it is real-valued and, for any permutation $\sigma$ in the symmetric group $\Sigma_n$, $f(t_1,\ldots,t_n) = f(t_{\sigma(1)},\ldots,t_{\sigma(n)})$ for almost every $t_1,\ldots,t_n\ge 0$ with respect to the product Lebesgue measure.
\end{itemize}
\end{definition}

An operator of the type $I^S(f)$ is self-adjoint if and only if $f$ is mirror symmetric.

\begin{definition} \label{def contraction}
Let $n,m$ be natural numbers, and let $f\in L^2(\RR^n_+)$ and $g\in L^2(\RR^m_+)$.  Let $p\le \min\{n,m\}$ be a natural number.  The {\it $p$th contraction} $f\overset{p}\frown g$ of $f$ and $g$ is the $L^2(\RR^{n+m-2p}_+)$ function defined by nested  integration of the middle $p$ variables in $f\tensor g$:
\begin{multline*}
 f\overset{p}\frown g\,(t_1,\ldots,t_{n+m-2p})\\ = \int_{\RR_+^{p}} f(t_1,\ldots,t_{n-p},s_1,\ldots,s_p)g(s_p,\ldots,s_1,t_{n-p+1},\ldots,t_{n+m-2p})\,ds_1\cdots ds_p. 
\end{multline*}
\end{definition}

\noindent Notice that when $p=0$, there is no integration, just the products of $f$ and $g$ with disjoint arguments; in other words, $f\overset{0}\frown g = f\tensor g$.

\subsection{Non-crossing Partitions}\label{ss:noncrossing}

A {\it partition} of $[n]=\{1,2,\dots,n\}$ is (as the name suggests) a collection of mutually disjoint nonempty subsets $B_1,$ $\ldots,B_r$ of $[n]$ such that $B_1\sqcup\cdots\sqcup B_r = [n]$.  The subsets are called the {\it blocks} of the partition. By convention we order the blocks by their least elements; i.e.\ $\min B_i < \min B_j$ iff $i<j$. If each block consists of two elements, then we call the partition a {\it pairing}. The set of all partitions on $[n]$ is denoted $\mathscr{P}(n)$, and the subset of all pairings is $\mathscr{P}_2(n)$.

\begin{definition} \label{def NC} Let $\pi\in\mathscr{P}(n)$ be a partition of $[n]$.  We say $\pi$ has a {\it crossing} if there are two distinct blocks $B_1,B_2$ in $\pi$ with elements $x_1,y_1\in B_1$ and $x_2,y_2\in B_2$ such that $x_1<x_2<y_1<y_2$.

\medskip

\noindent If $\pi\in\mathscr{P}(n)$ has no crossings, it is said to be a {\it non-crossing partition}.  The set of non-crossing partitions of $[n]$ is denoted $NC(n)$.  The subset of non-crossing pairings is denoted $NC_2(n)$.

\end{definition}

\begin{definition} \label{def respect} Let $n_1,\ldots,n_r$ be positive integers with $n=n_1+\cdots+n_r$.  The set $[n]$ is then partitioned accordingly as $[n] = B_1\sqcup\cdots\sqcup B_r$ where $B_1 = \{1,\ldots,n_1\}$, $B_2 = \{n_1+1,\ldots,n_1+n_2\}$, and so forth through $B_r = \{n_1+\cdots+n_{r-1}+1,\ldots,n_1+\cdots+n_r\}$.  Denote this partition as $n_1\tensor\cdots\tensor n_r$.

\medskip

\noindent We say that a pairing $\pi\in \mathscr{P}_2(n)$ {\it respects $n_1\tensor\cdots\tensor n_r$} if no block of $\pi$ contains more than one element from any given block of $n_1\tensor\cdots\tensor n_r$. The set of such respectful pairings is denoted $\mathscr{P}_2(n_1\tensor\cdots\tensor n_r)$.   The set of non-crossing pairings that respect $n_1\tensor\cdots\tensor n_r$ is denoted $NC_2(n_1\tensor\cdots\tensor n_r)$.
\end{definition}

\begin{definition} \label{def connected} Let $n_1,\ldots,n_r$ be positive integers, and let $\pi\in \mathscr{P}_2(n_1\tensor\cdots\tensor n_r)$.  Let $B_1$, $B_2$ be two blocks in $n_1\tensor\cdots\tensor n_r$.  Say that {\it $\pi$ links $B_1$ and $B_2$} if there is a block $\{i,j\}\in\pi$ such that $i\in B_1$ and $j\in B_2$.

\medskip

\noindent Define a graph $C_\pi$ whose vertices are the blocks of $n_1\tensor\cdots\tensor n_r$; $C_\pi$ has an edge between $B_1$ and $B_2$ iff $\pi$ links $B_1$ and $B_2$.  Say that {\it $\pi$ is connected} with respect to $n_1\tensor\cdots\tensor n_r$ (or that {\it $\pi$ connects the blocks of $n_1\tensor\cdots\tensor n_r$}) if the graph $C_\pi$ is connected. We shall denote by $NC_2^c(n_1\tensor\cdots\tensor n_r)$ the set of all non-crossing pairings that both respect and connect $n_1\tensor\cdots\tensor n_r$.
\end{definition}

\begin{definition} \label{def pairing integral} Let $n$ be an even integer, and let $\pi\in\mathscr{P}_2(n)$.  Let $f\colon \RR_+^n\to\CC$ be measurable. The {\it pairing integral} of $f$ with respect to $\pi$, denoted $\int_\pi f$, is defined (when it exists) to be the constant
\[ \int_\pi f = \int f(t_1,\ldots,t_n) \prod_{\{i,j\}\in\pi} \delta(t_i-t_j)\,dt_1\cdots dt_n. \]
\end{definition}

We finally introduce the notion of a semicircular family (see e.g. \cite[Definition 8.15]{NicaSpeicherBook}).

\begin{definition}\label{d:semifam} Let $d\geq 2$ be an integer, and let $c = \{c(i,j) : i,j=1,...,d\}$ be a positive definite symmetric matrix. A $d$-dimensional vector $(s_1,...,s_d)$ of random variables in $\mathscr{A}$ is said to be a {\it semicircular family with covariance} $c$ if for every $n\geq 1$ and every $(i_1,...,i_n)\in [d]^n$
\[
\varphi(s_{i_1}s_{i_2}\cdots s_{i_n}) = \sum_{\pi\in NC_2(n)}\prod_{\{a,b\}\in \pi} c(i_a,i_b).
\]
The previous relation implies in particular that, for every $i=1,...,d$, the random variable $s_i$ has the $S(0,c(i,i))$ distribution -- see Definition \ref{d:semicircular}.
\end{definition}

For instance, one can rephrase the defining property of the free Brownian motion $S=\{S_t : t\geq 0\}$ by saying that, for every $t_1<t_2<\cdots < t_d$, the vector $(S_{t_1},S_{t_2}-S_{t_1},...,S_{t_d}-S_{t_{d-1}})$ is a semicircular family with a diagonal covariance matrix such that $c(i,i) = t_i - t_{i-1}$ (with $t_0=0$), $i=1,...,d$.

\section{Proof of the main results}\label{s:proofs}
A crucial ingredient in the proof of Theorem \ref{t:main} is the following statement, showing that contractions control all important pairing integrals. This is the generalization of Proposition 2.2. in \cite{KNPS} to
our situation.

\begin{proposition}\label{prop:vanishingintegrals}
Let $d\geq 2$ and $n_1,\dots,n_d$ be some fixed positive integers.
Consider, for each $i=1,\dots,d$, sequences of mirror-symmetric functions $(f_k^{(i)})_{k\in\NN}$  
with $f_k^{(i)}\in L^2(\RR^{n_i}_+)$, satisfying:
\begin{itemize}
 \item 
There is a constant $M>0$ such that $\Vert f_k^{(i)}\Vert_{L^2(\RR^{n_i}_+)}\leq M$
for all $k\in\NN$ and all $i=1,\dots,d$.
\item
For all $i=1,\dots, d$ and all  $p=1,\dots,n_i-1$,
$$\lim_{k\to\infty} f_k^{(i)}\overset{p}\frown f_k^{(i)}=0\qquad \text{in}\quad L^2(\RR_+^{2n_i-2p}).$$

\end{itemize}

Let $r\geq3$, and let $\pi$ be a connected non-crossing pairing that respects
$n_{i_1}\otimes\cdots\otimes n_{i_r}$: $\pi\in NC_2^c(n_{i_1}\otimes\cdots\otimes n_{i_r})$.
Then
$$\lim_{k\to\infty} \int_\pi f_k^{(i_1)}\otimes\cdots\otimes
f_k^{(i_r)}=0.$$

\end{proposition}

\begin{proof}
In the same way as in \cite{KNPS} one sees that without restriction (i.e., up to a cyclic rotation and relabeling of the indices) one can
assume that
$$\int_\pi f_k^{(i_1)}\otimes\cdots\otimes
f_k^{(i_r)}=\int_{\pi'} (f_k^{(i_1)}\overset{p} \frown f_k^{(i_2)})\otimes
(f_k^{(i_3)}\otimes\cdots\otimes
f_k^{(i_r)}),
$$
where $0< 2p< n_{i_1}+n_{i_2}$ and
$$\pi'\in 
NC_2^c\bigl((n_{i_1}+n_{i_2}-2p)\otimes n_{i_3}\otimes\cdots\otimes n_{i_r}\bigr).$$
Note that
$0< 2p< n_{i_1}+n_{i_2}$ says that
$f_k^{(i_1)}\overset{p} \frown f_k^{(i_2)}$ is not a trivial contraction (trivial means that either nothing or all arguments are contracted); of course, in the case $n_{i_1}\not=n_{i_2}$ it is allowed that $p=\min(n_{i_1},n_{i_2})$.

By Lemma 2.1. of \cite{KNPS} we have then
\begin{align*}
\vert\int_\pi &f_k^{(i_1)}\otimes\cdots\otimes
f_k^{(i_r)}\vert\\
&\leq
\Vert f_k^{(i_1)}\overset{p} \frown f_k^{(i_2)}\Vert_{L^2(\RR_+^{n_{i_1}+n_{i_2}-2p})}\cdot
\Vert f_k^{(i_3)}\Vert_{L^2(\RR_+^{n_{i_3}})}\cdots
\Vert f_k^{(i_r)}\Vert_{L^2(\RR_+^{n_{i_r}})}\\
&\leq \Vert f_k^{(i_1)}\overset{p} \frown f_k^{(i_2)}\Vert_{L^2(\RR_+^{n_{i_1}+n_{i_2}-2p})}\cdot M^{r-2}.
\end{align*}
Now we only have to observe that, by also using the mirror symmetry of
$f_k^{(i_1)}$ and $f_k^{(i_2)}$, we have
\begin{align*}
&\Vert f_k^{(i_1)}\overset{p} \frown f_k^{(i_2)}\Vert_{L^2(\RR_+^{n_{i_1}+n_{i_2}-2p})}^2
=\left\la f_k^{(i_1)}\overset{n_{i_1}-p}\frown f_k^{(i_1)}, f_k^{(i_2)}\overset{n_{i_2}-p}\frown
f_k^{(i_2)}\right\ra_{L^2(\RR_+^{2p})}\\
&\leq
\Vert
f_k^{(i_1)}\overset{n_{i_1}-p}\frown f_k^{(i_1)}\Vert_{L^2(\RR_+^{2p})}\cdot
\Vert
 f_k^{(i_2)}\overset{n_{i_2}-p}\frown
f_k^{(i_2)}\Vert_{L^2(\RR_+^{2p})}.
\end{align*}
According to our assumption we have, for each $i=1,\dots,d$ and each $q=1,\dots,n_i-1$, that
$$\lim_{k\to\infty} f_k^{(i)}\overset{q}\frown f_k^{(i)} =0\qquad\text{in}
\qquad L^2(\RR_+^{2n_i-2q}).$$
Since now at least one of the two contractions $\overset{n_{i_1}-p}\frown$ and
$\overset{n_{i_2}-p}\frown$ is non-trivial, we can choose either $q=n_{i_1}-p$, $i=i_1$ or $q=n_{i_2}-p$, $i=i_2$ in the above, and this implies that
$$\lim_{k\to\infty}\Vert f_k^{(i_1)}\overset{p} \frown f_k^{(i_2)}\Vert_{L^2(\RR_+^{n_{i_1}+n_{i_2}-2p})}=0,$$
which gives our claim.
\end{proof}

We can now provide a complete proof of Theorem \ref{t:main}.

\begin{proof}[Proof of Theorem \ref{t:main}]

The equivalence between (2) and (3) follows from \cite{KNPS}. Clearly, (1) implies (3), so we only have the prove the reverse implication. So let us assume (3). Note that,
by Theorem 1.6 of \cite{KNPS}, this is equivalent to the fact that all non-trivial contractions of $f_k^{(i)}$ 
converge to 0; i.e., for each $i=1,\dots,d$ and each $q=1,\dots,n_i-1$ we have
\begin{equation}\label{eq:contractionsgotozero}
\lim_{k\to\infty} f_k^{(i)}\overset{q}\frown f_k^{(i)} =0\qquad\text{in}
\qquad L^2(\RR_+^{2n_i-2q}).
\end{equation} 
We will use statement (3) in this form. In order to show (1), we have to show that any moment in the variables 
$I(f_k^{(1)}),\dots,I(f_k^{(d)})$ converges, as $k\to\infty$, to the corresponding moment in the semicircular variables 
$s_1,\dots,s_d$. So, for $r\in\NN$ and positive integers $i_1,\dots,i_r$, we consider the moments
$$\varphi\left[I^S(f_k^{(i_1)})\cdots I^S(f_k^{(i_r)})\right].$$
We have to show that they converge, for $k\to\infty$, to the corresponding moment
$\ff(s_{i_1}\cdots s_{i_r})$.
Note that our assumption \eqref{thm:assumption} says that
$$\lim_{k \to\infty}\varphi[I^S(f_k^{(i)})I^S(f_k^{(j)})]=c(i,j)=\ff(s_is_j).$$
By Proposition 1.38 in \cite{KNPS} we have
$$\varphi\left[I^S(f_k^{(i_1)})\cdots I^S(f_k^{(i_r)})\right]=\sum_{\pi\in NC_2(n_{i_1}\otimes\cdots\otimes n_{i_r})}\int_\pi f_k^{(i_1)}\otimes\cdots\otimes
f_k^{(i_r)}.$$

By Remark 1.33 in \cite{KNPS}, any $\pi\in NC_2(n_{i_1}\otimes\cdots\otimes n_{i_r})$ can be uniquely decomposed into a disjoint union of connected pairings $\pi=\pi_1\sqcup\cdots \sqcup\pi_m$ with $\pi_q\in NC_2^c(\bigotimes_{j\in I_q} n_{i_j})$, where $\{1,\dots,r\}=I_1\sqcup\dots\sqcup I_m$ is a partition of the index set $\{ 1,\dots,r\}$. The above integral with respect to $\pi$ factors then accordingly
into
$$\int_\pi f_k^{(i_1)}\otimes\cdots\otimes
f_k^{(i_r)}=\prod_{q=1}^m \int_{\pi_q} \text{$\bigotimes_{j\in I_q} f_k^{(i_j)}$}.$$

Consider now one of those factors, corresponding to $\pi_q$. Since $\pi_q$
must respect $\bigotimes_{j\in I_q} n_{i_j}$, the number $r_q:=\#I_q$ must be strictly greater than 1. On the other hand, if $r_q\geq 3$, then, from \eqref{eq:contractionsgotozero} and Proposition \ref{prop:vanishingintegrals}, it follows that the corresponding pairing integral  
$\int_{\pi_q} \text{$\bigotimes_{j\in I_q} f_k^{(i_j)}$}$ converges to 0 in $L^2$.
Thus, in the limit, only those $\pi$ make a contribution, for which all $r_q$ are equal to 2, i.e., where each of the $\pi_q$ in the decomposition of $\pi$ 
corresponds to
a complete contraction between two of the appearing functions.
Let $NC_2^2(n_{i_1}\otimes\cdots\otimes n_{i_r})$ denote the set of those pairings $\pi$.
So we get
$$\lim_{k\to\infty}\varphi\left[I(f_k^{(i_1)})\cdots I(f_k^{(i_r)})\right]=
\sum_{\pi\in NC_2^2(n_{i_1}\otimes\cdots\otimes n_{i_r})} \lim_{k\to\infty}\int_\pi
f_k^{(i_1)}\otimes\cdots\otimes f_k^{(i_r)},$$
We continue as in \cite{KNPS}: each $\pi\in NC_2^2(n_{i_1}\otimes\cdots\otimes n_{i_r})$ is in bijection with a non-crossing pairing $\sigma\in NC_2(r)$. The contribution of such a $\pi$ is the product of the complete contractions for each pair of the corresponding $\sigma\in NC_2(r)$; but the complete contraction is just the $L^2$ inner product between the paired functions, i.e.,
$$\lim_{k\to\infty}\varphi\left[I^S(f_k^{(i_1)})\cdots I^S(f_k^{(i_r)})\right]=
\sum_{\sigma\in NC_2(r)} \prod_{\{s,t\}\in\sigma} c(i_s,i_t).
$$
This is exactly the moment
$\ff(s_{i_1}\cdots s_{i_r})$
of a semicircular family $(s_1,\dots,s_d)$ with covariance matrix $c$, and the proof is concluded.
\end{proof}

We conclude this paper with the proof of Theorem \ref{t:main2}.\\

\noindent
{\it Proof of Theorem \ref{t:main2}}. Point (1) is a simple consequence of the Wigner isometry (\ref{iso})
(since each $f_k^{(i)}$ is fully symmetric, $f_k^{(i)}$ is in particular mirror-symmetric),
together with the classical Wiener isometry which states that
\begin{equation}\label{iso}
E[I^W(f)I^W(g)] ={\bf 1}_{n=m} \times n!\langle f,g\rangle_{L^2(\RR_+^n)}
\end{equation}
for every $n,m\geq 0$, and every 
$f\in L^2(\mathbb{R}_+^n)$, $g\in L^2(\mathbb{R}_+^m)$.
For point (2), we observe first that the case $d=1$ is already known, as it corresponds to \cite[Theorem 1.8]{KNPS}.
Consider now the case $d\geq 2$.
Let us suppose that $\big(I^S(f_k^{(1)}),\dots,I^S(f_k^{(d)})\big)\overset{\rm law}{\to}(s_1,\dots,s_d)$.
In particular, $I^S(f_k^{(i)})\overset{\rm law}{\to}s_i$ for all $i=1,\ldots,d$. By 
\cite[Theorem 1.8]{KNPS} (case $d=1$), this implies that $I^W(f_k^{(i)})\overset{\rm law}{\to}\sqrt{(n_i)!}N_i$.
Since the asymptotic relations in (1) are verified, Theorem \ref{t:NPPT}{\bf (B)} leads then to
$\big(I^W(f_k^{(1)}),\dots,I^W(f_k^{(d)})\big)\overset{\rm law}{\to}\big(\sqrt{(n_i)!}N_1,\dots,\sqrt{(n_d)!}N_d\big)$,
which is the desired conclusion.
The converse implication follows exactly the same lines,and the proof is concluded.\qed

\end{document}